\documentclass[10pt]{article}

\usepackage{amsmath,amsbsy,amssymb,latexsym,amsthm,dsfont}
\usepackage{a4wide}

\newtheorem{Theorem}{Theorem}

\newtheorem{Lemma}{Lemma}

\newenvironment{keywords}{\begin{center}
\begin{minipage}[c]{13.2cm} {\bf Keywords:}} {\end{minipage}
\end{center}}

\newenvironment{msc}{\begin{center}
\begin{minipage}[c]{13.2cm} {\bf Mathematics Subject Classification 2010:}} {\end{minipage}
\end{center}}

\newenvironment{pacs}{\begin{center}
\begin{minipage}[c]{13.2cm} {\bf PACS:}} {\end{minipage}
\end{center}}

\begin{document}

\title{A Fractional Calculus of Variations for Multiple Integrals 
with Application to Vibrating String\thanks{Accepted for publication 
in the \emph{Journal of Mathematical Physics} (14/January/2010).}}

\author{Ricardo Almeida$^1$\\
\texttt{ricardo.almeida@ua.pt}
\and Agnieszka B. Malinowska$^{1, 2}$\\
\texttt{abmalinowska@ua.pt}
\and Delfim F. M. Torres$^1$\\
\texttt{delfim@ua.pt}}

\date{$^1$Department of Mathematics,
University of Aveiro,
3810-193 Aveiro, Portugal\\[0.3cm]
$^2$Faculty of Computer Science,
Bia{\l}ystok University of Technology,\\
15-351 Bia\l ystok, Poland}

\maketitle


\begin{abstract}
We introduce a fractional theory of the calculus of variations
for multiple integrals. Our approach uses the recent
notions of Riemann-Liouville fractional derivatives
and integrals in the sense of Jumarie. Main results provide
fractional versions of the theorems of Green and  Gauss,
fractional Euler-Lagrange equations,
and  fractional natural boundary conditions.
As an application we discuss the fractional equation
of motion of a vibrating string.
\end{abstract}


\begin{msc}
49K10; 26A33; 26B20.
\end{msc}



\section{Introduction}

The Fractional Calculus (FC) is one of the most interdisciplinary
fields of mathematics, with many applications in physics and engineering.
The history of FC goes back more than three centuries,
when in 1695 the derivative of order $\alpha=1/2$ was
described by Leibniz. Since then, many
different forms of fractional operators were introduced:
the Grunwald-Letnikov, Riemann-Liouville, Riesz, and Caputo
fractional derivatives \cite{Kilbas,Podlubny,samko},
and the more recent notions of Klimek \cite{Klimek},
Cresson \cite{Cresson}, and Jumarie \cite{Jumarie1,Jumarie2,Jumarie4,Jumarie3}.

FC is nowadays the realm of physicists and mathematicians, who investigate
the usefulness of such non-integer order derivatives and integrals
in different areas of physics and mathematics \cite{Carpinteri,Hilfer,Kilbas}.
It is a successful tool for describing complex quantum field dynamical systems,
dissipation, and long-range phenomena that cannot be well
illustrated using ordinary differential and integral operators
\cite{El-Nabulsi:Torres:2008,Hilfer,Klimek,rie}.
Applications of FC are found, \textrm{e.g.},
in classical and quantum mechanics, field theories,
variational calculus, and optimal control
\cite{El-Nabulsi:Torres:2007,Frederico:Torres2,Jumarie4}.

Although FC is an old mathematical discipline, the fractional vector calculus
is at the very beginning. We mention the recent paper \cite{Tarasov},
where some fractional versions of the classical results of Green, Stokes
and Gauss are obtained via Riemann-Liouville and Caputo fractional operators.
For the purposes of a multidimensional Fractional Calculus of Variations (FCV),
the Jumarie fractional integral and derivative seems, however,
to be more appropriate.

The FCV started in 1996 with the work of Riewe \cite{rie}. Riewe formulated the problem of the calculus
of variations with fractional derivatives and obtained the
respective Euler-Lagrange equations, combining both conservative and
nonconservative cases. Nowadays the FCV is a subject under strong research.
Different definitions for fractional derivatives and integrals
are used, depending on the purpose under study.
Investigations cover problems
depending on Riemann-Liouville fractional derivatives (see, \textrm{e.g.},
\cite{Atanackovic,El-Nabulsi:Torres:2008,Frederico:Torres1}), the Caputo
fractional derivative (see, \textrm{e.g.}, \cite{AGRA,Baleanu1,Baleanu2}),
the symmetric fractional derivative (see, \textrm{e.g.}, \cite{Klimek}), the
Jumarie fractional derivative (see, \textrm{e.g.}, \cite{Almeida,Jumarie4}),
and others \cite{Ric:Del:09,Cresson,El-Nabulsi:Torres:2007}.
For applications of the fractional calculus of variations we refer the reader to
\cite{Dreisigmeyer1,Dreisigmeyer2,El-Nabulsi:Torres:2008,Jumarie4,Klimek,Rabei2,Rabei1,Stanislavskya}.
Although the literature of FCV is already vast, much remains to be done.

Knowing the importance and relevance of multidimensional problems of the calculus of variations
in physics and engineering \cite{Weinstock},
it is at a first view surprising that a multidimensional FCV is a completely
open research area. We are only aware of some preliminary results presented
in \cite{El-Nabulsi:Torres:2008}, where it is claimed
that an appropriate fractional variational theory involving multiple integrals would have
important consequences in mechanical problems involving dissipative systems
with infinitely many degrees of freedom, but where a formal theory for that is missing.
There is, however, a good reason for such omission in the literature:
most of the best well-known fractional operators are not suitable
for a generalization of the FCV to the multidimensional case,
due to lack of good properties, \textrm{e.g.},
an appropriate Leibniz rule.

The main aim of the present work is to introduce a fractional calculus
of variations for multiple integrals. For that we make use
of the recent Jumarie fractional integral and derivative
\cite{Jumarie1,Jumarie2,Jumarie3}, extending such notions to the multidimensional case.
The main advantage of using Jumarie's approach lies in the following facts:
the Leibniz rule for the Jumarie fractional derivative is equal to the standard one
and, as we show, the fractional generalization of some fundamental multidimensional
theorems of calculus is possible. We mention that
Jumarie's approach is also useful for the one-dimensional FCV,
as recently shown in \cite{Almeida} (see also \cite{ref}).

The plan of the paper is as follows. In Section~\ref{sec:Prel} some basic
formulas of Jumarie's fractional calculus are briefly
reviewed. Then, in Section~\ref{sec:IT}, the differential and integral
vector operators are introduced, and fractional Green's, Gauss's and
Stokes' theorems formulated. Section~\ref{sec:NecOpt} is devoted
to the study of problems of fractional calculus of variations with
multiple integrals. Our main results provide Euler-Lagrange
necessary optimality type conditions for such problems
(Theorems~\ref{Theorem1} and \ref{Theorem3}) as well as
natural boundary conditions (Theorem~\ref{Theorem2}). We end with
Section~\ref{app} of applications and future perspectives.


\section{Preliminaries}
\label{sec:Prel}

For an introduction to the classical fractional calculus
we refer the reader to \cite{Kilbas,miller,Podlubny,samko}.
In this section we briefly review the main notions and results
from the recent fractional calculus proposed by Jumarie \cite{Jumarie2,Jumarie4,Jumarie3}.
Let $f:[0,1]\to\mathbb R$ be a continuous function and
$\alpha\in(0,1)$. The Jumarie fractional derivative of $f$
may be defined by
\begin{equation}
\label{eq:def:fd:Jumarie}
f^{(\alpha)}(x)=\frac{1}{\Gamma(1-\alpha)}\frac{d}{dx}\int_0^x(x-t)^{-\alpha}(f(t)-f(0))\,dt.
\end{equation}
One can obtain \eqref{eq:def:fd:Jumarie}
as a consequence of a more basic definition, a local one,
in terms of a fractional finite difference
(\textrm{cf.} equation (2.2) of \cite{Jumarie3}):
$$
f^{(\alpha)}(x) = \lim_{h\rightarrow 0}\frac{1}{h^\alpha}
\sum_{k=0}^{\infty} (-1)^k {\alpha \choose k} f\left(x+(\alpha-k) h\right)\, .
$$
Note that the Jumarie and the Riemann-Liouville fractional
derivatives are equal if $f(0)=0$. The advantage of definition \eqref{eq:def:fd:Jumarie}
with respect to the classical definition of Riemann-Liouville
is that the fractional derivative of a constant is now zero, as desired.
An anti-derivative of $f$, called the $(dt)^\alpha$
integral of $f$, is defined by
$$\int_0^xf(t)(dt)^\alpha=\alpha\int_0^x(x-t)^{\alpha-1}f(t)dt.$$
The following equalities can be considered
as fractional counterparts of the first and the second fundamental
theorems of calculus and can be found in \cite{Jumarie4,Jumarie3}:
\begin{equation}
\label{FTC1}
\frac{d^\alpha}{dx^\alpha}\int_0^xf(t)(dt)^\alpha=\alpha ! f(x),
\end{equation}
\begin{equation}
\label{FTC2}
\int_0^xf^{(\alpha)}(t)(dt)^\alpha=\alpha ! (f(x)-f(0)),
\end{equation}
where $\alpha!:=\Gamma(1+\alpha)$.
The Leibniz rule for the Jumarie fractional derivative is equal to
the standard one:
$$(f(x)g(x))^{(\alpha)}=f^{(\alpha)}(x)g(x)+f(x)g^{(\alpha)}(x).$$
Here we see another advantage of derivative \eqref{eq:def:fd:Jumarie}:
the fractional derivative of a product is not an infinite sum,
in opposite to the Leibniz rule for the Riemann-Liouville
fractional derivative \cite[p.~91]{Podlubny}.

One can easily generalize the previous definitions and results for
functions with a domain $[a,b]$:
$$f^{(\alpha)}(x)=\frac{1}{\Gamma(1-\alpha)}\frac{d}{dx}\int_a^x(x-t)^{-\alpha}(f(t)-f(a))\,dt$$
and
$$\int_a^xf(t)(dt)^\alpha=\alpha\int_a^x(x-t)^{\alpha-1}f(t)dt.$$


\section{Fractional integral theorems}
\label{sec:IT}

In this section we introduce some useful fractional integral
and fractional differential operators. With them we prove
fractional versions of the integral theorems
of Green, Gauss, and Stokes. Throughout the text we assume that all integrals and derivatives exist.


\subsection{Fractional operators}
\label{sec:Def}

Let us consider a continuous function $f = f\left(x_1,\ldots,x_n\right)$
defined on $R = \Pi_{i=1}^{n} [a_i,b_i]\subset\mathbb R^n$. Let us extend Jumarie's fractional derivative and the $(dt)^\alpha$
integral to functions with $n$ variables.
For $x_i\in[a_i,b_i]$, $i = 1,\ldots,n$, and $\alpha\in(0,1)$,
we define the fractional integral operator as
\begin{equation*}
{{_{a_i}}I_{x_i}^\alpha}[i]=\alpha \int_{a_i}^{x_i} (x_i-t)^{\alpha-1}\,dt \, .
\end{equation*}
These operators act on $f$ in the following way:
$${{_{a_i}}I_{x_i}^\alpha}[i] f(x_1,\ldots,x_n)=\alpha \int_{a_i}^{x_i} f(x_1,\ldots,x_{i-1},t,x_{i+1},\ldots,x_n)(x_i-t)^{\alpha-1}\,dt \, , \quad
i = 1,\ldots,n\, .
$$
Let $\Xi = \{k_1,\ldots,k_s\}$ be an arbitrary nonempty subset of
$\{1,\ldots,n\}$. We define the fractional multiple integral operator over the region
$R_\Xi = \Pi_{i=1}^{s}[a_{k_i},x_{k_i}]$ by
\begin{equation*}
\begin{split}
{I_{R_\Xi}^\alpha}[k_1,\ldots,k_s]
&={{_{a_{k_1}}}I_{x_{k_1}}^{\alpha}}[k_1] \ldots \,
{{_{a_{k_s}}}I_{x_{k_s}}^{\alpha}}[k_s]\\
&= \alpha^s \int_{a_{k_1}}^{x_{k_1}}\cdots\int_{a_{k_s}}^{x_{k_s}}
(x_{k_1}-t_{k_1})^{\alpha-1}\cdots(x_{k_s}-t_{k_s})^{\alpha-1}\,dt_{k_s} \ldots dt_{k_1}
\end{split}
\end{equation*}
which acts on $f$ by
\begin{multline*}
{I_{R_\Xi}^\alpha}[k_1,\ldots,k_s] f(x_1,\ldots,x_n)\\
=\alpha^s \int_{a_{k_1}}^{x_{k_1}}\cdots\int_{a_{k_s}}^{x_{k_s}}
f(\xi_1,\ldots,\xi_n) (x_{k_1}-t_{k_1})^{\alpha-1}\cdots(x_{k_s}-t_{k_s})^{\alpha-1}\,dt_{k_s} \ldots dt_{k_1}\, ,
\end{multline*}
where $\xi_j = t_j$ if $j \in \Xi$, and $\xi_j = x_j$ if $j \notin \Xi$, $j = 1,\ldots,n$.
The fractional volume integral of $f$ over the whole domain
$R$ is given by
$${I_R^\alpha} f= \alpha^n  \int_{a_1}^{b_1}\cdots\int_{a_n}^{b_n} f(t_1,\ldots,t_n) (b_1-t_1)^{\alpha-1}\cdots(b_n-t_n)^{\alpha-1}\,dt_n \ldots dt_1.$$

The fractional partial derivative operator with respect to
the $i$th variable $x_i$, $i = 1,\ldots,n$, of order $\alpha\in(0,1)$ is defined as follows:
$${{_{a_i}}D_{x_i}^{\alpha}}[i]=\frac{1}{\Gamma(1-\alpha)}\frac{\partial}{\partial x_i}\int_{a_i}^{x_i}(x_i-t)^{-\alpha}\,dt \, ,$$
which act on $f$ by
\begin{multline*}
{{_{a_i}}D_{x_i}^{\alpha}}[i] f(x_1,\ldots,x_n)\\
=\frac{1}{\Gamma(1-\alpha)}\frac{\partial}{\partial x_i}\int_{a_i}^{x_i}(x_i-t)^{-\alpha}\left[
f(x_1,\ldots,x_{i-1},t,x_{i+1},\ldots,x_n) - f(x_1,\ldots,x_{i-1},a_i,x_{i+1},\ldots,x_n)\right]\,dt \, ,
\end{multline*}
$i = 1,\ldots,n$.

We observe that the Jumarie fractional integral and the Jumarie fractional derivative can be obtained putting $n = 1$:
$${_aI_x^\alpha} [1]f(x)=\alpha \int_a^x (x-t)^{\alpha-1}f(t)\,dt= \int_a^x f(t)\,(dt)^\alpha$$
and
$${_aD_x^{\alpha}}[1]f(x)=\frac{1}{\Gamma(1-\alpha)}\frac{d}{d x}\int_a^x(x-t)^{-\alpha}(f(t)-f(a))\,dt=f^{(\alpha)}(x).$$
Using these notations, formulae \eqref{FTC1}--\eqref{FTC2} can be
presented as
$$
{_aD_x^{\alpha}}[1]{_aI_{x}^\alpha} [1]f(x)=\alpha!f(x)
$$
\begin{equation}
\label{FTCJ}
{_aI_x^\alpha} [1]{_aD_{x}^{\alpha}}[1]f(x)=\alpha!(f(x)-f(a)).
\end{equation}

In the two dimensional case we define the fractional line integral on $\partial R$, $R = [a,b]\times [c,d]$, by
$$
{I_{\partial R}^\alpha} f = {I_{\partial R}^\alpha} [1]f + {I_{\partial R}^\alpha} [2]f
$$
where
\begin{equation*}
\begin{split}
{I_{\partial R}^\alpha}[1]f &= {{_{a}}I_b^\alpha}[1] [ f(b,c) - f(b,d)]\\
&=\alpha \int_a^b [f(t,c)-f(t,d)]\,(b-t)^{\alpha-1}\,dt
\end{split}
\end{equation*}
and
\begin{equation*}
\begin{split}
{I_{\partial R}^\alpha}[2]f  &= {{_c}I_d^\alpha}[2] [ f(b,d)- f(a,d)]\\
   &=\alpha \int_c^d [f(b,t)-f(a,t)] \, (d-t)^{\alpha-1}\,dt.
\end{split}
\end{equation*}


\subsection{Fractional differential vector operations}
\label{sec:FDVO}

Let $W_X= [a,x]\times[c,y]\times[e,z]$, $W= [a,b]\times[c,d]\times[e,f]$, and denote
$(x_1,x_2,x_3)$ by $(x,y,z)$. We introduce the fractional nabla operator by
\begin{equation*}
{\nabla_{W_X}^{\alpha}}=i {_aD_x^{\alpha}}[1]+j {_cD_y^{\alpha}}[2]+k
{_eD_z^{\alpha}}[3]\, ,
\end{equation*}
where the $i$, $j$, $k$ define a fixed right-handed orthonormal basis.
If $f:\mathbb{R}^3\rightarrow \mathbb{R}$ is a continuous function,
then we define its fractional gradient as
\begin{equation*}
{\mbox{Grad}_{W_X}^{\alpha}}f={\nabla_{W_X}^{\alpha}}f
=i {_aD_x^{\alpha}}[1]f(x,y,z) + j {_cD_y^{\alpha}}[2]f(x,y,z) + k
{_eD_z^{\alpha}}[3]f(x,y,z) \, .
\end{equation*}
If $F=[F_x,F_y,F_z]:\mathbb{R}^3\rightarrow \mathbb{R}^3$ is a
continuous vector field, then we define its fractional divergence
and fractional curl by
\begin{equation*}
{\mbox{Div}_{W_X}^{\alpha}}F={\nabla_{W_X}^{\alpha}}\circ
F={_aD_x^{\alpha}}[1]F_x(x,y,z)+{_cD_y^{\alpha}}[2]F_y(x,y,z)+{_eD_z^{\alpha}}[3]F_z(x,y,z)
\end{equation*}
and
\begin{multline*}
{\mbox{Curl}_{W_X}^{\alpha}}F={\nabla_{W_X}^{\alpha}}\times
F=i\left({_cD_y^{\alpha}}[2]F_z(x,y,z)-{_eD_z^{\alpha}}[3]F_y(x,y,z)\right)\\
+j\left({_eD_z^{\alpha}}[3]F_x(x,y,z)-{_aD_x^{\alpha}}[1]F_z(x,y,z)\right)
+k\left({_aD_x^{\alpha}}[1]F_y(x,y,z)-{_cD_y^{\alpha}}[2]F_x(x,y,z)\right).
\end{multline*}
Note that these fractional differential operators are non-local.
Therefore, the fractional gradient, divergence, and curl,
depend on the region $W_X$.

For $F:\mathbb R^3\rightarrow \mathbb R^3$ and
$f,g:\mathbb R^3\rightarrow \mathbb R$ it is easy to check the
following relations:
\begin{itemize}
\item[(i)] ${\mbox{Div}_{W_X}^{\alpha}}(fF)=f{\mbox{Div}_{W_X}^{\alpha}}F+F\circ{\mbox{Grad}_{W_X}^{\alpha}}f$,
\item[(ii)] ${\mbox{Curl}_{W_X}^{\alpha}}({\mbox{Grad}_{W_X}^{\alpha}}f)=[0,0,0]$,
\item[(iii)] ${\mbox{Div}_{W_X}^{\alpha}}({\mbox{Curl}_{W_X}^{\alpha}}F)=0$,
\item[(iv)] ${\mbox{Grad}_{W_X}^{\alpha}}(fg)=g{\mbox{Grad}_{W_X}^{\alpha}}f+f{\mbox{Grad}_{W_X}^{\alpha}}g$,
\item[(v)] ${\mbox{Div}_{W_X}^{\alpha}}({\mbox{Grad}_{W_X}^{\alpha}}f)={_aD_x^{\alpha}}[1]{_aD_x^{\alpha}}[1]f+
{_cD_y^{\alpha}}[2]{_cD_y^{\alpha}}[2]f+
{_eD_z^{\alpha}}[3]{_eD_z^{\alpha}}[3]f$.
\end{itemize}
Let us recall that in general $({D_{W_X}^{\alpha}})^2\neq {D_{W_X}^{2 \alpha}}$ (see \cite{Jumarie3}).

A fractional flux of the vector field $F$ across $\partial W$
is a fractional oriented surface integral of the field such that
\begin{equation*}
({I_{\partial W}^{\alpha}},F)={I_{\partial W}^{\alpha}}[2,3]F_x(x,y,z)
+{I_{\partial W}^{\alpha}}[1,3]F_y(x,y,z)+{I_{\partial W}^{\alpha}}[1,2]F_z(x,y,z) \, ,
\end{equation*}
where
$${I_{\partial W}^{\alpha}}[1,2]f(x,y,z)={{_{a}}I_b^\alpha}[1] {{_{c}}I_d^\alpha}[2]  [ f(b,d,f) - f(b,d,e)] \, ,$$
$${I_{\partial W}^{\alpha}}[1,3]f(x,y,z)={{_{a}}I_b^\alpha}[1] {{_{e}}I_f^\alpha}[3]  [ f(b,d,f) - f(b,c,f)] \, ,$$
and
$${I_{\partial W}^{\alpha}}[2,3]f(x,y,z)={{_{c}}I_d^\alpha}[2] {{_{e}}I_f^\alpha}[3]  [ f(b,d,f) - f(a,d,f)] \, .$$


\subsection{Fractional theorems of Green, Gauss, and Stokes}
\label{sec:Green}

We now formulate the fractional formulae of Green, Gauss, and Stokes.
Analogous results via Caputo fractional derivatives
and Riemann-Liouville fractional integrals were obtained
by Tarasov in \cite{Tarasov}.

\begin{Theorem}[Fractional Green's theorem for a rectangle]
Let $f$ and $g$ be two continuous functions whose domains contain
$R=[a,b]\times[c,d]\subset\mathbb R^2$. Then,
$$
{I_{\partial R}^\alpha} [1]f+{I_{\partial R}^\alpha} [2]g
=\frac{1}{\alpha!}{I_{R}^\alpha}\left[{_aD_b^{\alpha}}[1]g-{_cD_d^{\alpha}}[2]f\right]\, .
$$
\end{Theorem}
\begin{proof}
We have
$$
{I_{\partial R}}^\alpha [1]f+{I_{\partial R}^\alpha} [2]g={_aI_b^\alpha}[1] [f(b,c)-f(b,d)] + {_cI_d^\alpha}[2]
[g(b,d)-g(a,d)]\, .
$$
By equation \eqref{FTCJ},
\begin{equation*}
\begin{split}
f(b,c)-f(b,d)&=-\frac{1}{\alpha!} {_cI_d^\alpha}[2]
{_cD_d^{\alpha}}[2]f(b,d)\, ,\\
g(b,d)-g(a,d)&=\frac{1}{\alpha!}
{_aI_b^\alpha}[1] {_aD_b^{\alpha}}[1]g(b,d)\, .
\end{split}
\end{equation*}
Therefore,
\begin{equation*}
\begin{split}
{I_{\partial R}^\alpha} [1]f+ {I^\alpha_{\partial R}} [2]g
&= - {_aI_b^\alpha}[1] \frac{1}{\alpha!} {_cI_d^\alpha}[2]
{_cD_d^{\alpha}}[2]f(b,d) + {_cI_d^\alpha}[2] \frac{1}{\alpha!}
{_aI_b^\alpha}[1] {_aD_b^{\alpha}}[1]g(b,d)\\
&=\frac{1}{\alpha!}{I_{R}^\alpha}\left[{_aD_b^{\alpha}}[1]g-{_cD_d^{\alpha}}[2]f\right]\, .
\end{split}
\end{equation*}
\end{proof}

\begin{Theorem}[Fractional Gauss's theorem for a parallelepiped]
\label{GaussTheo}
Let $F=(F_x,F_y,F_z)$ be a continuous vector field in a domain that contains
$W=[a,b]\times[c,d]\times[e,f]$. If the boundary of $W$ is a closed
surface $\partial W$, then
\begin{equation}
\label{Gauss}
({I_{\partial W}^{\alpha}},F)=\frac{1}{\alpha !}
{I_{W}^{\alpha}}\mbox{Div}{_W^{\alpha}}F\, .
\end{equation}
\end{Theorem}
\begin{proof}
The result follows by direct transformations:
\begin{equation*}
\begin{split}
({I_{\partial W}^{\alpha}},F)&={I^\alpha_{\partial W}} [2,3]F_x+{I^\alpha_{\partial W}} [1,3]F_y+{I^\alpha_{\partial W}} [1,2]F_z\\
&={_cI_d^\alpha}[2] {_eI_f^\alpha}[3](F_x(b,d,f)-F_x(a,d,f))+ {_aI_b^\alpha}[1] {_eI_f^\alpha}[3](F_y(b,d,f)-F_y(b,c,f))\\
&\qquad +{_aI_b^\alpha}[1]{_cI_d^\alpha}[2](F_z(b,d,f)-F_z(b,d,e))\\
&= \frac{1}{\alpha!} {_aI_b^\alpha}[1] {{_c}I_d^\alpha}[2]{{_e}I_f^{\alpha}}[3]
({_aD_b^{\alpha}}[1]F_x(b,d,f)+{_cD_d^{\alpha}}[2]F_y(b,d,f)+{_eD_f^{\alpha}}[3]F_z(b,d,f))\\
&=\frac{1}{\alpha!} {I_W^\alpha}({_aD_b^{\alpha}}[1]F_x+{_cD_d^{\alpha}}[2]F_y+{_eD_f^{\alpha}}[3]F_z)\\
&=\frac{1}{\alpha !}{I_{W}^{\alpha}}{\mbox{Div}_W^{\alpha}}F\, .
\end{split}
\end{equation*}
\end{proof}

Let $S$ be an open, oriented, and nonintersecting surface,
bounded by a simple and closed curve $\partial S$.
Let $F=[F_x,F_y,F_z]$ be a continuous vector field.
Divide up $S$ by sectionally curves into
$N$ subregions $S_1$, $S_2$, $\ldots$, $S_N$. Assume that for small enough
subregions each $S_j$ can be approximated by a plane rectangle
$A_j$ bounded by curves $C_1$, $C_2$, $\ldots$, $C_N$. Apply Green's
theorem to each individual rectangle $A_j$. Then, summing over the
subregions,
\begin{equation*}
\sum_j\frac{1}{\alpha !}{I_{A_j}^{\alpha}}({\nabla_{A_j}^{\alpha}}\times
F)=\sum_jI_{\partial A_j}^{\alpha}F \, .
\end{equation*}
Furthermore, letting $N\rightarrow\infty$
\begin{equation*}
\sum_j\frac{1}{\alpha !}{I_{A_j}^{\alpha}}({\nabla_{A_j}^{\alpha}}\times
F)\rightarrow \frac{1}{\alpha
!}({I_{S}^{\alpha}},{\mbox{Curl}_{S}^{\alpha}} F)
\end{equation*}
while
\begin{equation*}
\sum_jI_{\partial A_j}^{\alpha}F\rightarrow {I_{\partial
S}^{\alpha}}F\, .
\end{equation*}
We conclude with the fractional Stokes formula
\begin{equation*}
\frac{1}{\alpha !}({I_{S}^{\alpha}},{\mbox{Curl}_{S}^{\alpha}}
F)={I_{\partial S}^{\alpha}}F\, .
\end{equation*}


\section{Fractional calculus of variations with
multiple integrals} \label{sec:NecOpt}

Consider a function $w=w(x,y)$ with two variables. Assume that the
domain of $w$ contains the rectangle $R=[a,b]\times[c,d]$ and that
$w$ is continuous on $R$. We introduce the variational functional defined by
\begin{equation}
\label{eq:funct:prb:P}
\begin{split}
J(w)&=I_{R}^\alpha L\left(x,y,w(x,y),{_aD_x^{\alpha}}[1]w(x,y),{_cD_y^{\alpha}}[2]w(x,y)\right)\\
   &:= \alpha^2 \displaystyle\int_a^b\int_c^d L\left(x,y,w,{_aD_x^{\alpha}}[1]w,{_cD_y^{\alpha}}[2]w\right)
   (b-x)^{\alpha-1}(d-y)^{\alpha-1}\,dydx.
\end{split}
\end{equation}
We assume that the lagrangian $L$ is at least of class $C^1$. Observe that,
using the notation of the $(dt)^\alpha$ integral as presented in \cite{Jumarie3},
\eqref{eq:funct:prb:P} can be written as
\begin{equation}
\label{eq:functP}
J(w)=\displaystyle\int_a^b\int_c^d
L\left(x,y,w(x,y),{_aD_x^{\alpha}}[1]w(x,y),{_cD_y^{\alpha}}[2]w(x,y)\right)\,(dy)^\alpha (dx)^\alpha\, .
\end{equation}
Consider the following FCV problem, which we address as problem $(P)$.\\

\textbf{Problem $\mathbf{(P)}$}: minimize (or maximize) functional $J$ defined by \eqref{eq:functP}
with respect to the set of continuous functions $w(x,y)$
such that $w |_{\partial R} = \varphi(x,y)$ for some given function $\varphi$.\\

The continuous functions $w(x,y)$ that assume the prescribed values $w |_{\partial R} = \varphi(x,y)$
at all points of the boundary curve of $R$ are said to be admissible.
In order to prove necessary optimality conditions for problem $(P)$
we use a two dimensional analogue of fractional integration by parts.
Lemma~\ref{lemma} provides the necessary fractional rule.

\begin{Lemma}
\label{lemma}
Let $F$, $G$, and $h$ be continuous functions whose domains contain $R$.
If $h\equiv 0$ on $\partial R$, then
\begin{multline*}
\int_a^b\int_c^d [G(x,y){_aD_x^{\alpha}}[1]h(x,y)-F(x,y){_cD_y^{\alpha}}[2]h(x,y)] (b-x)^{\alpha-1}(d-y)^{\alpha-1}\,dydx\\
=-\int_a^b\int_c^d [({_aD_x^{\alpha}}[1]G(x,y)-{_cD_y^{\alpha}}[2]F(x,y))h(x,y)] (b-x)^{\alpha-1}(d-y)^{\alpha-1}\,dydx\, .
\end{multline*}
\end{Lemma}
\begin{proof}
By choosing $f=F\cdot h$ and $g=G\cdot h$ in Green's formula, we obtain
$${I_{\partial R}^\alpha} [1](F \, h)+{I_{\partial R}^\alpha} [2](G \, h)
=\frac{1}{\alpha!}{I_{R}^\alpha}[{_aD_b^{\alpha}}[1]G \cdot h+G\cdot {_aD_b^{\alpha}}[1]h
-{_cD_d^{\alpha}}[2]F \cdot h-F\cdot {_cD_d^{\alpha}}[2]h] \, ,$$
which is equivalent to
$$
\frac{1}{ \alpha!}{I_{R}^\alpha} \left[G \cdot {_aD_b^{\alpha}}[1]h -F \cdot {_cD_d^{\alpha}}[2]h \right]
={I_{\partial R}^\alpha} [1](F \, h)+{I_{\partial R}^\alpha} [2](G \, h)-\frac{1}{\alpha!}{I_{R}^\alpha}\left[
({_aD_b^{\alpha}}[1]G- {_cD_d^{\alpha}}[2]F ) h\right]\, .$$
In addition, since $h \equiv 0$ on $\partial R$, we deduce that
$${I_{R}^\alpha} \left[G \cdot {_aD_b^{\alpha}}[1]h-F \cdot {_cD_d^{\alpha}}[2]h \right]=-{I_{R}^\alpha}
\left[({_aD_b^{\alpha}}[1]G-{_cD_d^{\alpha}}[2]F)h \right] \, .$$
The lemma is proved.
\end{proof}

\begin{Theorem}[Fractional Euler-Lagrange equation]
\label{Theorem1}
Let $w$ be a solution to problem $(P)$.
Then $w$ is a solution of the fractional partial differential equation
\begin{equation}
\label{NecEquation}
\partial_3 L - {_aD_x^{\alpha}}[1] \partial_4 L - {_cD_y^{\alpha}}[2] \partial_5 L =0\, ,
\end{equation}
where by $\partial_i L$, $i = 1, \ldots, 5$, we denote the
usual partial derivative of $L(\cdot,\cdot,\cdot,\cdot,\cdot)$
with respect to its $i$-th argument.
\end{Theorem}
\begin{proof}
Let $h$ be a continuous function on $R$ such that $h\equiv 0$ on $\partial R$,
and consider an admissible variation $w+\epsilon h$,
for $\epsilon$ taking values on a sufficient small neighborhood of zero. Let
$$j(\epsilon)=J(w+\epsilon h).$$
Then $j'(0)=0$, \textrm{i.e.},
$$\alpha^2 \displaystyle\int_a^b\int_c^d\left(\partial_3 L \, h
+ \partial_4 L \,{_aD_x^{\alpha}}[1]h + \partial_5 L \,{_cD_y^{\alpha}}[2]h\right)(b-x)^{\alpha-1}(d-y)^{\alpha-1}\,dydx=0\, .$$
Using Lemma~\ref{lemma}, we obtain
$$\alpha^2 \int_a^b\int_c^d\left(\partial_3 L - {_aD_x^{\alpha}}[1] \partial_4 L
- {_cD_y^{\alpha}}[2] \partial_5 L \right)h (b-x)^{\alpha-1}(d-y)^{\alpha-1}\,dydx=0\, .$$
Since $h$ is an arbitrary function, by the fundamental lemma of calculus of variations we deduce equation \eqref{NecEquation}.
\end{proof}

Let us consider now the situation where we do not impose admissible functions $w$ to be of fixed values on $\partial R$.\\

\textbf{Problem $\mathbf{(P')}$}: minimize (or maximize) $J$ among the set of all
continuous curves $w$ whose domain contains $R$.\\

\begin{Theorem}[Fractional natural boundary conditions]
\label{Theorem2}
Let $w$ be a solution to problem $(P')$. Then $w$ is a solution of the fractional differential equation
\eqref{NecEquation} and satisfies the following equations:
\begin{enumerate}
\item $\partial_4L(a,y,w(a,y),{_aD_a^{\alpha}}[1]w(a,y),{_cD_y^{\alpha}}[2]w(a,y))=0$ for all $y \in [c,d]$;
\item $\partial_4L(b,y,w(b,y),{_aD_b^{\alpha}}[1]w(b,y),{_cD_y^{\alpha}}[2]w(b,y))=0$ for all $y \in [c,d]$;
\item $\partial_5L(x,c,w(x,c),{_aD_x^{\alpha}}[1]w(x,c),{_cD_c^{\alpha}}[2]w(x,c))=0$ for all $x \in [a,b]$;
\item $\partial_5L(x,d,w(x,d),{_aD_x^{\alpha}}[1]w(x,d),{_cD_d^{\alpha}}[2]w(x,d))=0$ for all $x \in [a,b]$.
\end{enumerate}
\end{Theorem}
\begin{proof} Proceeding as in the proof of Theorem~\ref{Theorem1} (see also Lemma~\ref{lemma}),
we obtain
\begin{multline}
\label{eq1}
0 =\alpha^2 \displaystyle\int_a^b\int_c^d\left(\partial_3 L \, h+ \partial_4 L \,
{_aD_x^{\alpha}}[1]h + \partial_5 L \, {_cD_y^{\alpha}}[2]h\right)(b-x)^{\alpha-1}(d-y)^{\alpha-1}\,dydx\\
=\alpha^2 \displaystyle\int_a^b\int_c^d\left(\partial_3 L
- {_aD_x^{\alpha}}[1]\partial_4 L- {_cD_y^{\alpha}}[2]\partial_5 L\right)h
\,(b-x)^{\alpha-1}(d-y)^{\alpha-1}\,dydx\\
+\alpha! {I_{\partial R}^\alpha} [2](\partial_4L \, h)- \alpha! {I_{\partial R}^\alpha} [1](\partial_5 L \, h)\, ,
\end{multline}
where $h$ is an arbitrary continuous function. In particular the
above equation holds for $h\equiv 0$ on $\partial R$. For such $h$
the second member of \eqref{eq1} vanishes and by the fundamental
lemma of the calculus of variations we deduce equation
\eqref{NecEquation}. With this result, equation \eqref{eq1} takes the
form
\begin{equation}
\label{eq:basia9}
\begin{array}{ll}
0 &= \displaystyle \int_c^d \partial_4L(b,y,w(b,y),{_aD_b^{\alpha}}[1]w(b,y),{_cD_y^{\alpha}}[2]w(b,y)) \, h(b,y)(d-y)^{\alpha-1}dy\\
  &\quad -\displaystyle \int_c^d \partial_4L(a,y,w(a,y),{_aD_a^{\alpha}}[1]w(a,y),{_cD_y^{\alpha}}[2]w(a,y)) \, h(a,y)(d-y)^{\alpha-1}dy\\
  &\quad -\displaystyle \int_a^b \partial_5L(x,c,w(x,c),{_aD_x^{\alpha}}[1]w(x,c),{_cD_c^{\alpha}}[2]w(x,c)) \, h(x,c)(b-x)^{\alpha-1}dx\\
  &\quad +\displaystyle \int_a^b \partial_5L(x,d,w(x,d),{_aD_x^{\alpha}}[1]w(x,d),{_cD_d^{\alpha}}[2]w(x,d)) \, h(x,d)(b-x)^{\alpha-1}dx\, .
\end{array}
\end{equation}
Since $h$ is an arbitrary function, we can consider the subclass of functions for which $h\equiv0$ on
$$[a,b]\times \{c\}\cup[a,b]\times \{d\}\cup \{b\} \times [c,d].$$
For such $h$ equation \eqref{eq:basia9} reduce to
$$0=\displaystyle \int_c^d \partial_4L(a,y,w(a,y),{_aD_a^{\alpha}}[1]w(a,y),{_cD_y^{\alpha}}[2]w(a,y)) \, h(a,y)(d-y)^{\alpha-1}dy.$$
By the fundamental lemma of calculus of variations, we obtain
$$
\partial_4L(a,y,w(a,y),{_aD_a^{\alpha}}[1]w(a,y),{_cD_y^{\alpha}}[2]w(a,y))=0 \quad \mbox{for all } y \in [c,d]\, .
$$
The other natural boundary conditions are proved similarly, by appropriate
choices of $h$.
\end{proof}

We can generalize Lemma~\ref{lemma} and Theorem~\ref{Theorem1} to
the three dimensional case in the following way.

\begin{Lemma}
\label{lemma2}
Let $A$, $B$, $C$, and $\eta$ be continuous functions whose domains contain the
parallelepiped $W$. If $\eta\equiv 0$ on $\partial W$, then
\begin{equation}
\label{parts3D_2}
{I_{W}^{\alpha}}(A \cdot {_aD_b^{\alpha}}[1]\eta+B  \cdot {_cD_d^{\alpha}}[2]\eta+C \cdot {_eD_f^{\alpha}}[3]\eta)
=-{I_{W}^{\alpha}}( \left[ {_aD_b^{\alpha}}[1]A+ {_cD_d^{\alpha}}[2]B+ {_eD_f^{\alpha}}[3]C\right] \eta )\, .
\end{equation}
\end{Lemma}
\begin{proof}
By choosing $F_x=\eta A$, $F_y=\eta B$, and $F_z=\eta C$ in
\eqref{Gauss}, we obtain the three dimensional analogue of
integrating by parts:
\begin{multline*}
{I_{W}^{\alpha}}(A \cdot {_aD_b^{\alpha}}[1]\eta+B\cdot {_cD_d^{\alpha}}[2]\eta+C\cdot {_eD_f^{\alpha}}[3]\eta)\\
= -{I_{W}^{\alpha}}(\left[ {_aD_b^{\alpha}}[1]A+ {_cD_d^{\alpha}}[2]B+{_eD_f^{\alpha}}[3]C\right] \eta )+\alpha !({I_{\partial W}^{\alpha}},[\eta A,\eta
B,\eta C]).
\end{multline*}
In addition, if we assume that $\eta\equiv 0$ on $\partial W$, we
have formula \eqref{parts3D_2}.
\end{proof}

\begin{Theorem}[Fractional Euler-Lagrange equation for triple integrals]
\label{Theorem3}
Let $w=w(x,y,z)$ be a continuous function whose domain contains $W=[a,b]\times[c,d]\times[e,f]$.
Consider the functional
\begin{equation*}
\begin{split}
J(w)&=I_{W}^\alpha L\left(x,y,z,w(x,y,z),{_aD_x^{\alpha}}[1]w(x,y,z),{_cD_y^{\alpha}}[2]w(x,y,z),{_eD_z^{\alpha}}[3]w(x,y,z)\right)\\
&= \int_a^b\int_c^d\int_e^f
L\left(x,y,z,w,{_aD_x^{\alpha}}[1]w,{_cD_y^{\alpha}}[2]w,{_eD_z^{\alpha}}[3]w\right)\,(dz)^\alpha (dy)^\alpha (dx)^\alpha
\end{split}
\end{equation*}
defined on the set of continuous curves such that their values on
${\partial W}$ take prescribed values. Let $L$ be at least of class $C^1$. If $w$ is a
minimizer (or maximizer) of $J$, then $w$ satisfies the fractional partial differential equation
$$\partial_4 L - {_aD_b^{\alpha}}[1] \partial_5 L - {_cD_d^{\alpha}}[2] \partial_6 L - {_eD_f^{\alpha}}[3] \partial_7 L =0.$$
\end{Theorem}
\begin{proof}
A proof can be done similarly to the proof of Theorem~\ref{Theorem1} where instead
of using Lemma~\ref{lemma} we apply Lemma~\ref{lemma2}.
\end{proof}


\section{Applications and possible extensions}
\label{app}

In classical mechanics, functionals that depend on functions of two or more variables
arise in a natural way, \textrm{e.g.}, in mechanical problems involving systems with
infinitely many degrees of freedom (string, membranes, etc.). Let us
consider a flexible elastic string stretched under constant tension
$\tau$ along the $x$ axis with its end points fixed at $x=0$ and
$x=L$. Let us denote the transverse displacement of the particle
at time $t$, $t_1\leq t \leq t_2$,
whose equilibrium position is characterized by its
distance $x$ from the end of the string at $x=0$ by the function
$w=w(x,t)$. Thus $w(x,t)$, with $0\leq x\leq L$, describes the shape of the string
during the course of the vibration. Assume a distribution of mass along
the string of density $\sigma=\sigma(x)$. Then the function that describes
the actual motion of the string is one which renders
$$J(w)=\frac12 \int_{t_1}^{t_2}\int_0^L(\sigma w_t^2-\tau w_x^2)\, dx \, dt$$
an extremum with respect to functions $w(x,t)$ which describe the actual configuration
at $t=t_1$ and $t=t_2$ and which vanish, for all $t$, at $x=0$ and $x=L$
(see \cite[p.~95]{Weinstock} for more details).

We discuss the description of the motion of the string within the
framework of the fractional differential calculus.
One may assume that, due to some constraints of physical nature,
the dynamics do not depend on the usual partial derivatives but on some
fractional derivatives ${_0D_x^{\alpha}}[1]w$ and
${{_{t_1}D}_t^{\alpha}}[2]w$. For example,
we can assume that there is some coarse graining phenomenon --
see details in \cite{ref,Jumarie3b}.
In this condition, one is entitled to
assume again that the actual motion of the system, according to the
principle of Hamilton, is such as to render the action function
$$J(w)=\frac12 I_{R}^\alpha (\sigma \, ({{_{t_1}}D_t^{\alpha}}[2]w)^2
-\tau \, ({_0D_x^{\alpha}}[1]w)^2),$$
where $R=[0,L]\times[t_1,t_2]$, an extremum. Note that
we recover the classical problem of the vibrating
string when $\alpha \rightarrow 1^{-}$. Applying Theorem~\ref{Theorem1}
we obtain the fractional equation of motion for the vibrating string:
$$
{_0D_x^{\alpha}}[1] {_0D_{1}^{\alpha}}[1]w
=\frac{\sigma }{\tau}{{_{t_1}D}_t^{\alpha}}[2]{{_{t_1}D}_{2}^{\alpha}}[2]w\, .
$$
This equation becomes the classical equation of the vibrating string
(\textrm{cf.}, \textrm{e.g.}, \cite[p.~97]{Weinstock}) if $\alpha \rightarrow 1^-$.

We remark that the fractional operators are non-local, therefore they are suitable
for constructing models possessing memory effect. In the above
example, we discussed the application of the fractional differential
calculus to the vibrating string. We started with a variational
formulation of the physical process in which we modify the
Lagrangian density by replacing integer order derivatives with
fractional ones. Then the action integral in the sense of Hamilton
was minimized and the governing equation of the physical process was
obtained in terms of fractional derivatives.
Similarly, many others
physical fields can be derived from a suitably defined action
functional. This gives several possible applications of the fractional
calculus of variations with multiple integrals as was introduced in
this paper, \textrm{e.g.}, in describing non-local properties of
physical systems in mechanics (see, \textrm{e.g.},
\cite{Baleanu4,Carpinteri,Klimek,Rabei2,Tarasov2}) or
electrodynamics (see, \textrm{e.g.}, \cite{Baleanu3,Tarasov}).

We end with some open problems for further investigations.
It has been recognized that fractional calculus is useful
in the study of scaling in physical systems \cite{Cresson:02,Cresson:Gasta:Delfim}.
In particular, there is a direct connection between local fractional differentiability
properties and the dimensions of Holder exponents of nowhere differentiable functions,
which provide a powerful tool to analyze the behavior of irregular signals and functions
on a fractal set \cite{Ric:Del:Hold,Kolwankar}.
Fractional calculus appear naturally, \textrm{e.g.},
when working with fractal sets and coarse-graining spaces
\cite{Jumarie3b,Jumarie5}, and fractal patterns of deformation
and vibration in porous media and heterogeneous materials \cite{[12]}.
The importance of vibrating strings to the fractional
calculus has been given in \cite{Zareba}, where it is shown
that a fractional Brownian motion can be identified with a string.
The usefulness of our fractional theory of the calculus of variations
with multiple integrals in physics, to deal with fractal
and coarse-graining spaces, porous media, and Brownian motions,
are questions to be studied. It should be possible to prove the theorems
obtained in this work for a general form of domains and boundaries;
and to develop a fractional calculus of variations with multiple integrals
in terms of other type of fractional operators. An interesting
open question consists to generalize the fractional Noether-type theorems
obtained in \cite{Atanackovic2,Frederico:Torres1,Frederico:Torres2} to the
case of several independent variables.


\section*{Acknowledgments}

Work supported by the {\it Centre for Research on Optimization and
Control} (CEOC) from the ``Funda\c{c}\~{a}o para a Ci\^{e}ncia e a
Tecnologia'' (FCT), cofinanced by the European Community Fund
FEDER/POCI 2010. Agnieszka Malinowska is also supported by
Bia{\l}ystok University of Technology,
via a project of the Polish Ministry of Science and Higher Education
``Wsparcie miedzynarodowej mobilnosci naukowcow''.

The authors are grateful to an anonymous referee for useful remarks
and references.



\end{document}